\documentclass[pdftex,11pt]{amsart} 
\usepackage[pdftex]{hyperref} 
\usepackage{amsmath,amssymb}
\usepackage{mathtools}
\usepackage{amscd}
\usepackage{mathrsfs}
\usepackage[foot]{amsaddr}
\usepackage{amsthm}
\usepackage{setspace}
\usepackage[all]{xy}
\usepackage{comment}
\usepackage[pdftex]{graphicx} 
\usepackage{tikz}
\usetikzlibrary{arrows}
\usetikzlibrary{shapes.geometric,fit}
\usepackage{tikz-cd}
\usepackage{fancyhdr}

\theoremstyle{definition}
\newtheorem{thm}{Theorem}[subsection]
\newtheorem*{thm*}{Theorem}
\newtheorem{defi}[thm]{Definition}
\newtheorem*{defi*}{Definition}
\newtheorem*{acknowledge}{Acknowledgement}
\newtheorem{cor}[thm]{Corollary}
\newtheorem*{cor*}{Corollary}
\newtheorem{prop}[thm]{Proposition}
\newtheorem*{prop*}{Proposition}
\newtheorem{lem}[thm]{Lemma}
\newtheorem*{lem*}{Lemma}
\newtheorem{ex}[thm]{Example}
\newtheorem*{ex*}{Example}

\newtheorem{rem}[thm]{Remark}
\newtheorem*{rem*}{Remark}

\newtheorem*{hw*}{Homework}

\newcommand{\C}{\mathbb{C}}

\newcommand{\Z}{\mathbb{Z}}
\newcommand{\N}{\mathbb{N}}
\newcommand{\T}{\mathbb{T}}
\DeclareMathOperator{\supp}{supp}

\newcommand{\Bis}{\mathrm{Bis}}
\DeclareMathOperator{\Iso}{\mathrm{Iso}}

\DeclareMathOperator{\dom}{\mathrm{dom}}
\DeclareMathOperator{\ab}{\mathrm{ab}}
\DeclareMathOperator{\id}{\mathrm{id}}
\DeclareMathOperator{\Aut}{\mathrm{Aut}}
\DeclareMathOperator{\FAut}{\mathrm{FAut}}

\def\i<#1>{\langle #1 \rangle}
\def\l<#1>{\left\langle #1 \right\rangle}

\makeatletter
\renewenvironment{proof}[1][\proofname]{\par
  \normalfont
  \topsep6\p@\@plus6\p@ \trivlist
  \item[\hskip\labelsep{\bfseries #1}\@addpunct{.}]\ignorespaces
}{%
  \endtrivlist
}
\renewcommand{\proofname}{\sc{Proof}}
\makeatother
\makeatletter
\newcommand*{\defeq}{\mathrel{\rlap{%
                     \raisebox{0.3ex}{$\m@th\cdot$}}%
                     \raisebox{-0.3ex}{$\m@th\cdot$}}%
                     =}
\makeatother

\title[]{*-homomorphisms between groupoid C*-algebras}
\author{Fuyuta Komura}
\address{Center for Advanced Intelligence Project, RIKEN,
	3-14-1 Hiyoshi, Kohoku-ku, Yokohama, 223-8522, Japan}
\address{Phone: +81-45-566-1641+42706}
\address{Fax: +81-45-566-1642}
\email{fuyuta.k@keio.jp}
\subjclass[2010]{20M18, 22A22, 46L05}

\begin{document}
\maketitle
\begin{abstract}
	
	In this paper, we investigate *-homomorphisms between C*-algebras associated to \'etale groupoids.
	First, we prove that such a *-homomorphism can be described by closed invariant subsets, groupoid homomorphisms and cocycles under some assumptions.
	Then we prove C*-rigidity results for \'etale groupoids which are not necessarily effective.
	As another application, we investigate certain subgroups of the automorphism groups of groupoid C*-algebras.
	More precisely,
	we show that the groups of automorphisms that globally preserve the function algebras on the unit spaces are isomorphic to certain semidirect product groups.
	As a corollary,
	we show that, if group actions on groupoid C*-algebras fix the function algebras on the unit spaces,
	then the actions factors through the abelianizations of the acting groups.

\end{abstract}

\section{Introduction}

The main subject in the present paper is C*-algebras associated with \'etale groupoids,
that is, groupoid C*-algebras.
The theory of groupoid C*-algebras is initiated by Renault in \cite{renault1980groupoid}.
It is known that many C*-algebras are realized as groupoid C*-algebras (see \cite{asims}, for example).
It is a natural task to characterize properties of groupoid C*-algebras in terms of \'etale groupoids.
For instance,
see \cite{anantharaman2000} for the relation between nuclearity of groupoid C*-algebras and amenability of topological groupoids.
In addition, simplicity of full groupoid C*-algebras is investigated in \cite{Brown2014}.
Recently, the authors in \cite{BrownExelFuller2021} established the Galois correspondence result between \'etale groupoids and twisted groupoid C*-algebras.
In \cite{Komura2022},
the author studied certain submodules in groupoid C*-algebras and analyse discrete group coactions on groupoid C*-algebras.
In the present paper,
we investigate *-homomorphisms between groupoid C*-algebras.
First of all,
we explain our motivation to study *-homomorphisms between groupoid C*-algebras.

For an \'etale groupoid $G$,
one can construct a (reduced) groupoid C*-algebra $C^*_r(G)$ and a commutative C*-subalgebra $C_0(G^{(0)})\subset C^*_r(G)$.
Renault proved a C*-rigidity result in \cite{renault}.
Namely,
for effective \'etale groupoids $G_1$ and $G_2$,
Renault proved that $G_1$ and $G_2$ are isomorphic as \'etale groupoids if inclusions of C*-algebras $C_0(G_1^{(0)})\subset C^*_r(G_1)$ and $C_0(G_2^{(0)})\subset C^*_r(G_2)$ are isomorphic.
This result connects the bridge between \'etale groupoids and C*-algebras and has various applications.
Indeed,
one can deduce a C*-rigidity result for C*-algebras associated with dynamical systems from Renault's result (see \cite{MatsumotoMatsui}, for example).
Now, assume that inclusions $C_0(G_1^{(0)})\subset C^*_r(G_1)$ and $C_0(G^{(0)}_2)\subset C^*_r(G_2)$ are isomorphic.
One may wonder how many isomorphisms between $C_0(G_1^{(0)})\subset C^*_r(G_1)$ and $C_0(G^{(0)}_2)\subset C^*_r(G_2)$ exist.
To solve this problem,
it is sufficient to determine the following group
\[\Aut_{C_0(G_1^{(0)})}(C^*_r(G_1))\defeq \{\varphi\in \Aut(C^*_r(G_1))\mid \varphi(C_0(G_1^{(0)}))=C_0(G_1^{(0)})\}.\]
Therefore, 
we are motivated to investigate $\Aut_{C_0(G^{(0)})}(C^*_r(G))$ for an effective \'etale groupoid $G$.
In Corollary \ref{cor: AutG is a semidirect product},
we prove that $\Aut_{C_0(G^{(0)})}(C^*_r(G))$ is isomorphic to the semidirect group of $\Aut (G)$ and $Z(G,\T)$,
where $Z(G,\T)$ is the abelian group of $\T$-valued $1$-cocycles on $G$.
We remark that the almost same result is obtained in \cite[Proposition 5.7 (1)]{Matui2011} and we obtain a slightly more direct proof by using our main theorem, that is, Theorem \ref{thm summary of main theorem }.
In addition, we shall remark that the similar result has been already obtained for von Neumann algebras arising from equivalence relations in \cite[Theorem 3]{FeldmanMoore}.
In the present paper,
we deal with an analogue of \cite[Theorem 3]{FeldmanMoore} for groupoid C*-algebras.

As stated above,
our purpose in this paper is to study *-homomorphisms between groupoid C*-algebras such as elements in $\Aut_{C_0(G^{(0)})}(C^*_r(G))$.
For example, a similar attempt succeeded in \cite[Theorem 6.3]{DonsigPitts2008} for isomorphisms between regular C*-inclusions in terms of coordinate systems.
In this paper,
we begin with the study of general *-homomorphisms between groupoid C*-algebras which need not to be isomorphisms.
Our main theorem is Theorem \ref{thm summary of main theorem },
which asserts that a *-homomorphism $\varphi\colon C^*_r(G)\to C^*_r(H)$ can be described in terms of underlying \'etale groupoids $G$ and $H$ under assumptions that $H$ is effective and $\varphi$ has some compatibility with $C_0(G^{(0)})$ and $C_0(H^{(0)})$.
Taking into account that previous works like \cite{renault} relies on the effectiveness of underlying \'etale groupoids,
it seems noteworthy that we do not assume the effectiveness of $G$.
As a direct application of Theorem \ref{thm summary of main theorem },
we prove that surjective *-homomorphisms between groupoid C*-algebras induce quotients of \'etale groupoids (Corollary \ref{cor : surjective *-hom induces a quotient of etale groupoid}).
Corollary \ref{cor : surjective *-hom induces a quotient of etale groupoid} generalizes Renault's result in \cite{renault},
which asserts that *-isomorphisms between groupoid C*-algebras which preserve the Cartan subalgebras induce isomorphisms between \'etale groupoids.
From Corollary \ref{cor : surjective *-hom induces a quotient of etale groupoid},
we obtain the following variant of the rigidity results for not necessarily effective \'etale groupoids :
for \'etale (not necessarily effective) groupoids $G$ which have closed $\Iso(G)^{\circ}$ and some amenability condition,
the quotient groupoids $G/\Iso(G)^{\circ}$ are invariants for the inclusion of C*-algebras $(C^*_r(G), C_0(G^{(0)}))$ (Corollary \ref{cor generalized rigidity result}).
In other words,
$(C^*_r(G), C_0(G^{(0)}))$ remembers $G/\Iso(G)^{\circ}$ even if an \'etale groupoid $G$ is not effective.

As another application of Theorem \ref{thm summary of main theorem },
we prove the structure theorem of $\Aut_{C_0(G^{(0)})}(C^*_r(G))$ (Corollary \ref{cor: AutG is a semidirect product}).
More precisely,
we show that $\Aut_{C_0(G^{(0)})}(C^*_r(G))$ is isomorphic to the semidirect product of $\Aut(G)$ and $Z(G,\T)$.
In particular,
it turns out that $Z(G,\T)$ corresponds to a certain abelian subgroup of $\Aut_{C_0(G^{(0)})}(C^*_r(G))$. 
In Corollary \ref{cor Z(G,T) is isomorphic to FAUT},
we prove that $Z(G,\T)$ corresponds to
\[\FAut_{C_0(G^{(0)})}(C^*_r(G))\defeq\{\varphi\in \Aut(C^*_r(G))\mid \text{$\varphi(a)=a$ for all $a\in C_0(G^{(0)})$}\}. \]
As a corollary,
we show that a group action on a groupoid C*-algebra factors through its abelianization if the fixed point subalgebra contains $C_0(G^{(0)})$ (Corollary \ref{cor action with large fixed point algebra factors abelianization}).

This paper is organized as follows.
In Section 1,
we recall fundamental facts about \'etale groupoids,
groupoid C*-algebras and inverse semigroups.
In Section 2, we prove our main theorems about *-homomorphisms between groupoid C*-algebras.
Our goal in the first subsection, Subsection \ref{subsection *-homomorphism between groupoid C*-algebras}, is Theorem \ref{thm summary of main theorem }.
Toward Theorem \ref{thm summary of main theorem },
we first prove that *-homomorphisms between groupoid C*-algebras induce groupoid homomorphisms and $\T$-valued $1$ cocycles (Lemma \ref{lem sigma is an equivalent map} and Lemma \ref{lem definition of cocyle}).
Then we prove that every *-homomorphisms can be described by closed invariant subsets,
groupoid homomorphisms and $\T$-valued $1$-cocycles (Proposition \ref{prop varphiU =varphic, Phi,U} and Proposition \ref{prop decompotision of varphi}).
As a special case of Theorem \ref{thm summary of main theorem },
we observe that surjective *-homomorphisms induce quotients of \'etale groupoids in Corollary \ref{cor : surjective *-hom induces a quotient of etale groupoid}.
This immediately implies a variant of rigidity result for not necessarily effective groupoids (Corollary \ref{cor generalized rigidity result}).
In Subsection \ref{subsection automorphism groups of C*r(G)},
we investigate $\Aut_{C_0(G^{(0)})}(C^*_r(G))$ for an effective \'etale groupoid $G$.
First, we prove that $\Aut_{C_0(G^{(0)})}(C^*_r(G))$ is isomorphic to the natural semidirect product of $\Aut(G)$ and $Z(G,\T)$ (Corollary \ref{cor: AutG is a semidirect product}).
Then we observe that $Z(G,\T)$ corresponds to $\FAut_{C_0(G^{(0)})}(C^*_r(G))$ (Corollary \ref{cor Z(G,T) is isomorphic to FAUT}).
As a by-product,
we show that a group action on a groupoid C*-algebra factors through its abelianization if the fixed point algebra contains $C_0(G^{(0)})$ (Corollary \ref{cor action with large fixed point algebra factors abelianization}).

\begin{acknowledge}
	The author would like to thank Prof.\ Takeshi Katsura for his support and encouragement.
	The author is also grateful to Yuki Arano for fruitful discussions.
	This work was supported by JST CREST Grant Number JPMJCR1913 and RIKEN Special Postdoctoral Researcher Program.
\end{acknowledge}

\section{Preliminaries}

In this section,
we recall fundamental notions about \'etale groupoids,
groupoid C*-algebras and inverse semigroups.

\subsection{\'Etale groupoids}

We recall the basic notions on \'etale groupoids.
See \cite{asims} and \cite{paterson2012groupoids} for more details.

A groupoid is a set $G$ together with a distinguished subset $G^{(0)}\subset G$,
domain and range maps $d,r\colon G\to G^{(0)}$ and a multiplication 
\[
G^{(2)}\defeq \{(\alpha,\beta)\in G\times G\mid d(\alpha)=r(\beta)\}\ni (\alpha,\beta)\mapsto \alpha\beta \in G
\]
such that
\begin{enumerate}
	\item for all $x\in G^{(0)}$, $d(x)=x$ and $r(x)=x$ hold,
	\item for all $\alpha\in G$, $\alpha d(\alpha)=r(\alpha)\alpha=\alpha$ holds,
	\item for all $(\alpha,\beta)\in G^{(2)}$, $d(\alpha\beta)=d(\beta)$ and $r(\alpha\beta)=r(\alpha)$ hold,
	\item if $(\alpha,\beta),(\beta,\gamma)\in G^{(2)}$,
	we have $(\alpha\beta)\gamma=\alpha(\beta\gamma)$,
	\item\label{inverse} every $\gamma \in G$,
	there exists $\gamma'\in G$ which satisfies $(\gamma',\gamma), (\gamma,\gamma')\in G^{(2)}$ and $d(\gamma)=\gamma'\gamma$ and $r(\gamma)=\gamma\gamma'$.   
\end{enumerate}
Since the element $\gamma'$ in (\ref{inverse}) is uniquely determined by $\gamma$,
$\gamma'$ is called the inverse of $\gamma$ and denoted by $\gamma^{-1}$.
We call $G^{(0)}$ the unit space of $G$.
A subgroupoid of $G$ is a subset of $G$ which is closed under the inversion and multiplication. 
For $U\subset G^{(0)}$, we define $G_U\defeq d^{-1}(U)$ and $G^{U}\defeq r^{-1}(U)$.
We define also $G_x\defeq G_{\{x\}}$ and $G^x\defeq G^{\{x\}}$ for $x\in G^{(0)}$.
A subset $F\subset G^{(0)}$ is said to be invariant if $d(\alpha)\in F$ implies $r(\alpha)\in F$ for all $\alpha\in G$.
If $F\subset G^{(0)}$ is invariant,
$G_F\subset G$ is a subgroupoid and the unit space of $G_F$ is $F$.

A topological groupoid is a groupoid equipped with a topology where the multiplication and the inverse are continuous.
A topological groupoid is said to be \'etale if the domain map is a local homeomorphism.
Note that the range map of an \'etale groupoid is also a local homeomorphism.
In this paper,
although there exist important \'etale groupoids that are not Hausdorff,
we assume that \'etale groupoids are always locally compact Hausdorff unless otherwise stated.
Hence, we mean locally compact Hausdorff \'etale groupoids by \'etale groupoids.

A subset $U$ of an \'etale groupoid $G$ is called a bisection if the restrictions $d|_U$ and $r|_U$ is injective.
If follows that $d|_U$ and $r|_U$ are homeomorphism onto their images if $U$ is a bisection since $d$ and $r$ are open maps.

An \'etale groupoid $G$ is said to be effective if $G^{(0)}$ coincides with the interior of $\Iso(G)$,
where 
\[
\Iso(G)\defeq\{\alpha\in G\colon d(\alpha)=r(\alpha)\}
\]
is the isotropy of $G$.
An \'etale groupoid $G$ is said to be topologically principal if 
\[
\{x\in G^{(0)}\mid G_x\cap G^x=\{x\}\}
\]
is dense in $G^{(0)}$.
If $G$ is topologically principal,
then $G$ is effective.
If $G$ is second countable and effective,
then $G$ is topologically principal (see \cite[Proposition 3.6]{renault}).

A groupoid homomorphism $c\colon G\to \T$ is called a $\T$-valued $1$-cocycle,
where $\T$ denotes the circle group.
Because we only consider a $\T$-valued $1$-cocycle in this paper,
we often simply call it a cocycle.
We let $Z(G,\T)$ denote the set of all continuous cocycles $c\colon G\to \T$.
Then $Z(G,\T)$ is an abelian group with respect to the pointwise product.

\subsection{Groupoid C*-algebras}

We recall the definition of groupoid C*-algebras.

Let $G$ be an \'etale groupoid.
Then $C_c(G)$, the vector space of compactly supported continuous $\C$-valued functions on $G$, is a *-algebra with respect to the multiplication and the involution defined by
\[
f*g(\gamma)\defeq\sum_{\alpha\beta=\gamma}f(\alpha)g(\beta), f^*(\gamma)\defeq\overline{f(\gamma^{-1})},
\]
where $f,g\in C_c(G)$ and $\gamma\in G$.
The left regular representation $\lambda_x\colon C_c(G)\to B(\ell^2(G_x))$ at $x\in G^{(0)}$ is defined by
\[
\lambda_x(f)\delta_{\alpha}\defeq \sum_{\beta\in G_{r(\alpha)}}f(\alpha)\delta_{\alpha\beta},
\]
where $f\in C_c(G)$ and $\alpha\in G_x$.
The reduced norm $\lVert\cdot\rVert_{r}$ on $C_c(G)$ is defined by
\[
\lVert f\rVert_r\defeq \sup_{x\in G^{(0)}} \lVert \lambda_x(f)\rVert
\]
for $f\in C_c(G)$.
We often omit the subscript `$r$' of $\lVert\cdot\rVert_r$ if there is no chance to confuse.
The reduced groupoid C*-algebra $C^*_r(G)$ is defined to be the completion of $C_c(G)$ with respect to the reduced norm.
Note that $C_c(G^{(0)})\subset C_c(G)$ is a *-subalgebra and this inclusion extends to the inclusion  $C_0(G^{(0)})\subset C^*_r(G)$.

For a closed invariant subset $F\subset G^{(0)}$,
the closed subgroupoid $G_F\subset G$ is \'etale with respect to the relative topology.
It is well-known that the restriction
\[
C_c(G)\ni f\mapsto f|_{G_F}\in C_c(G_F)
\]
extends to the surjective *-homomorphism $C^*_r(G)\to C^*_r(G_F)$.
In addition,
the reduced groupoid C*-algebra $C^*_r(G)$ can be embedded into $C_0(G)$ as in the following.
See \cite[Proposition 9.3.3]{asims} for the proof.

\begin{prop}[Evaluation] \label{prop evaluation}
	Let $G$ be an \'etale groupoid.
	For $a\in C^*_r(G)$,
	$j(a)\in C_0(G)$ is defined by
	\[
	j(a)(\alpha)\defeq\i<\delta_{\alpha}|\lambda_{d(\alpha)}(a)\delta_{d(\alpha)}>
	\]
	for $\alpha\in G$\footnote{In this paper, inner products of Hilbert spaces are linear with respect to the right variables.}.
	Then $j\colon C^*_r(G)\to C_0(G)$ is a norm decreasing injective linear map.
	Moreover, $j$ is an identity map on $C_c(G)$.
	
\end{prop}

\begin{rem}
	Since $j\colon C^*_r(G)\to C_0(G^{(0)})$ is injective,
	we may identify $j(a)$ with $a$.
	Hence, we often regard $a$ as a function on $G$ and simply denote $j(a)$ by $a$.
\end{rem}

Finally, we recall facts about normalizers.
	
\begin{defi}
	Let $A$ be a C*-algebra and $D\subset A$ be a C*-subalgebra.
	An element $n\in A$ is called a normalizer for $D$ if $nDn^*\cup n^*Dn\subset D$ holds.
	We denote the set of normalizers for $D$ by $N(A, D)$.
\end{defi}	
	
	For $a\in C^*_r(G)$,
	we denote the open support of $a$ by
	\[
	\supp^{\circ}(a)\defeq \{\alpha\in G\mid a(\alpha)\not=0\}.
	\]
	Note that $\supp^{\circ}(a)$ is open in $G$.
	Normalizers for $C_0(G^{(0)})$ and bisections in $G$ are intimately related as follows.

\begin{prop}[{\cite[Proposition 4.8]{renault}}] \label{prop open support of normaliser is a bisection}
	Let $G$ be an \'etale groupoid and $U\subset G$ be an open set.
	If $U$ is a bisection, then every elements in $C_c(U)$ is a normalizer.
	Moreover, if $n\in C^*_r(G)$ is a normalizer and $G$ is effective,
	then $\supp^{\circ}(n)\subset G$ is an open bisection.
\end{prop}

	\subsection{Inverse semigroups}
	
	We recall the basic notions about inverse semigroups.
	See \cite{lawson1998inverse} or \cite{paterson2012groupoids} for more details.
	An inverse semigroup $S$ is a semigroup where for every $s\in S$ there exists a unique $s^*\in S$ such that $s=ss^*s$ and $s^*=s^*ss^*$.
	We denote the set of all idempotents in $S$ by $E(S)\defeq\{e\in S\mid e^2=e\}$.
	It is known that $E(S)$ is a commutative subsemigroup of $S$.
	An inverse semigroup which consists of idempotents is called a (meet) semilattice of idempotents.
	A zero element is a unique element $0\in S$ such that $0s=s0=0$ holds for all $s\in S$.
	An inverse semigroup with a unit is called an inverse monoid.
	By a subsemigroup of $S$,
	we mean a subset of $S$ that is closed under the product and inverse of $S$.
	A map $\varphi\colon S\to T$ between inverse semigroups $S$ and $T$ is called a semigroup homomorphism if $\varphi(st)=\varphi(s)\varphi(t)$ holds for all $s,t\in S$.
	Note that a semigroup homomorphism automatically preserves generalised inverses (i.e.\ $\varphi(s^*)=\varphi(s)^*$ holds for all $s\in S$).
	
	For a topological space $X$,
	we denote by $I_X$ the set of all homeomorphisms between open sets in $X$.
	Then $I_X$ is an inverse semigroup with respect to the product defined by the composition of maps.
	For an inverse semigroup $S$,
	an inverse semigroup action $\alpha\colon S\curvearrowright X$ is a semigroup homomorphism $S\ni s\mapsto \alpha_s\in I_X$.
	In this paper, we always assume that every action $\alpha$ satisfies $\bigcup_{e\in E(S)}\dom(\alpha_e)=X$.
	
	\subsection{Inverse semigroups associated to inclusions of C*-algebras}
		
	Following \cite[Proposition 13.3]{noncommutativeCartanExel},
	we associate inverse semigroups of slices to inclusions of C*-algebras.
	
\begin{defi}
	Let $D\subset A$ be an inclusion of C*-algebras.
	A slice  is a norm closed subspace $M\subset A$ such that $DM\cup MD\subset M$ and $M\subset N(A, D)$.
	The set of all slices is denoted by $\mathcal{S}(A,D)$.
\end{defi}

\begin{prop}[{\cite[Proposition 13.3]{noncommutativeCartanExel}}]
	Let $D\subset A$ be an inclusion of C*-algebras.
	Assume that $D$ has an approximate unit for $A$.
	For $M,N\in \mathcal{S}(A,D)$,
	define $MN$ to be the closure of the linear span of
	\[
	\{xy\in A\mid x\in M, y\in N \}.
	\]
	Then $\mathcal{S}(A,B)$ is an inverse semigroup under this operation.
	The generalized inverse of $M\in \mathcal{S}(A, D)$ is $M^*\defeq\{x^*\in A\mid x\in M\}$.
\end{prop}

Let $G$ be an \'etale groupoid and $\mathrm{Bis}(G)$ denotes the set of all open bisections in $G$.
For $U, V\in \mathrm{Bis}(G)$,
their product is defined by
\[
UV\defeq \{\alpha\beta\in G\mid \alpha\in U, \beta\in V, d(\alpha)=r(\beta)\}.
\]
Then $UV\in\mathrm {Bis}(G)$ and $\mathrm{Bis}(G)$ is an inverse semigroup with respect to this product.
Note that $U^*\in \mathrm{Bis}(G)$ is given by
\[
U^{-1}\defeq \{\alpha^{-1}\in G\mid \alpha\in U\}.
\]

For $U\in \Bis(G)$ and $f\in C_c(U)$,
we have $f^*f\in C_0(G^{(0)})$ and 
\[\lVert f\rVert_r^2=\lVert f^*f\rVert_r=\sup_{x\in G^{(0)}}\lvert f^*f(x)\rvert=\sup_{\alpha\in U}\lvert f(\alpha)\rvert^2. \]
Hence the reduced norm coincides with the supremum norm on $C_c(U)$ and we may identify the closure of $C_c(U)$ in $C^*_r(G)$ with $C_0(U)$.
Note that $C_0(U)\subset C^*_r(G)$ is a $C_0(G^{(0)})$-subbimodule.

Now we associate a $C_0(G^{(0)})$-subbimodule $C_0(U)\subset C^*_r(G)$ to a bisection $U\in \Bis(G)$.
This gives a semigroup homomorphism as in the following.

\begin{thm}[{\cite[Proposition 1.4.3., Corollary 2.1.6.]{Komura2022}}] \label{thm homomorphism from Bis to slices is isomorphism}
	Let $G$ be an \'etale groupoid.
	Then the map 
	\[\Psi\colon \mathrm{Bis}(G)\ni U\mapsto C_0(U)\in \mathcal{S}(C^*_r(G), C_0(G^{(0)})) \]
	is an injective semigroup homomorphism.
	If $G$ is effective,
	then $\Psi$ is an isomorphism.
\end{thm}

\subsection{\'Etale groupoids associated to inverse semigroup actions}

Many \'etale groupoids arise from actions of inverse semigroups to topological spaces.
We recall how to construct an \'etale groupoid from an inverse semigroup action.

Let $X$ be a locally compact Hausdorff space.
Recall that $I_X$ is the inverse semigroup of homeomorphisms between open sets in $X$.
For $e\in E(S)$, we denote the domain of $\alpha_e$ by $D_e^{\alpha}$.
Then $\alpha_s$ is a homeomorphism from $D_{s^*s}^{\alpha}$ to $D_{ss^*}^{\alpha}$.
We often omit $\alpha$ of $D_{e}^{\alpha}$ if there is no chance to confuse.

For an action $\alpha\colon S\curvearrowright X$,
we associate an \'etale groupoid $S\ltimes_{\alpha}X$ as the following.
First we put the set $S*X\defeq \{(s,x) \in S\times X \mid x\in D^{\alpha}_{s^*s}\}$.
Then we define an equivalence relation $\sim$ on $S*X$ by declaring that $(s,x)\sim (t,y)$ holds if
\[
\text{$x=y$ and there exists $e\in E(S)$ such that $x\in D^{\alpha}_e$ and $se=te$}.  
\]
Set $S\ltimes_{\alpha}X\defeq S*X/{\sim}$ and denote the equivalence class of $(s,x)\in S*X$ by $[s,x]$.
The unit space of $S\ltimes_{\alpha}X$ is $X$, where $X$ is identified with the subset of $S\ltimes_{\alpha}X$ via the injective map
\[
X\ni x\mapsto [e,x] \in S\ltimes_{\alpha}X, x\in D^{\alpha}_e.
\]
The domain map and range maps are defined by
\[
d([s,x])=x, r([s,x])=\alpha_s(x)
\]
for $[s,x]\in S\ltimes_{\alpha}X$.
The product of $[s,\alpha_t(x)],[t,x]\in S\ltimes_{\alpha}X$ is $[st,x]$.
The inverse is $[s,x]^{-1}=[s^*,\alpha_s(x)]$.
Then $S\ltimes_{\alpha}X$ is a groupoid in these operations.
For $s\in S$ and an open set $U\subset D_{s^*s}^{\alpha}$,
define 
\[[s, U]\defeq \{[s,x]\in S\ltimes_{\alpha}X\mid x\in U\}.\]
These sets form an open basis of $S\ltimes_{\alpha}X$.
In these structures,
$S\ltimes_{\alpha}X$ is a locally compact \'etale groupoid,
although $S\ltimes_{\alpha}X$ is not necessarily Hausdorff.
In this paper,
we only treat inverse semigroup actions $\alpha\colon S\curvearrowright X$ such that $S\ltimes_{\alpha}X$ become Hausdorff.

\begin{ex}\label{example : action of inverse semigroups of bisections}
	Let $G$ be an \'etale groupoid.
	For $U\in \Bis(G)$,
	put $\theta_U=r|_U\circ d|_U^{-1}$.
	Then $\theta_U\colon d(U)\to r(U)$ is a homeomorphism and we obtain an action $\theta\colon \Bis(G)\curvearrowright G^{(0)}$.
	We call this action the canonical action of $\Bis(G)$.
	By \cite[Proposition 5.4]{ExelcombinatrialC*algebra},
	$G$ is isomorphic to $\Bis(G)\ltimes_{\theta}G^{(0)}$.
	Indeed,
	the map
	\[
	\Phi\colon \Bis(G)\ltimes_{\alpha}G^{(0)}\ni [U,x]\mapsto \alpha\in G
	\]
	is an isomorphism,
	where $\alpha$ is the unique element in $U$ such that $d(\alpha)=x$.
\end{ex}

We will use the following proposition to construct a groupoid homomorphism.
The proof is left to the readers.
\begin{prop}\label{proposition equivariant map induces groupoid homomorphism}
	Let $\alpha\colon S\curvearrowright X$ and $\beta\colon T\curvearrowright Y$ be actions of inverse semigroups $S$ and $T$ on topological spaces $X$ and $Y$.
	Assume that a continuous map $\sigma\colon X\to Y$ and a semigroup homomorphism $\psi\colon S\to T$ satisfies the following condition : 
	\begin{center}
		If $x\in X$ and $s\in S$ satisfies $x\in D^{\alpha}_{s^*s}$,
		then $\sigma(x)\in D^{\beta}_{\psi(s^*s)}$ and $\beta_{\psi(s)}(\sigma(x))=\sigma(\alpha_s(x))$ hold.
	\end{center}
	Then the map
	\[\Phi\colon S\ltimes_{\alpha} X\ni [s,x]\to [\psi(s), \sigma(x)]\in T\ltimes_{\beta}Y \]
	is a continuous groupoid homomorphism.
	If $\sigma\colon X\to Y$ is locally homeomorphic,
	then $\Phi$ is also locally homeomorphic.
\end{prop}

\section{Main theorems}

\subsection{*-homomorphisms between groupoid C*-algebras}\label{subsection *-homomorphism between groupoid C*-algebras}

In this subsection, we investigate *-homomorphisms between groupoid C*-algebras.
Our goal in this subsection is to show the next theorem.

\begin{thm}\label{thm summary of main theorem }
	Let $G$ be an \'etale groupoid and $H$ be an \'etale effective groupoid.
	Assume that we are given a *-homomorphism $\varphi\colon C^*_r(G)\to C^*_r(H)$ such that $\varphi(C_0(G^{(0)}))\subset C_0(H^{(0)})$ holds and $\varphi(C_0(G^{(0)}))$ is an ideal of $C_0(H^{(0)})$.
	Then there exist
	\begin{enumerate}
		\item a closed invariant subset $F\subset G^{(0)}$,
		\item a locally homeomorphic groupoid homomorphism $\Phi\colon G_F\to H$ such that $\Phi|_{U\cap G_F}$ is a homeomorphism onto its image for each $U\in \Bis(G)$\footnote{In particular, $\Phi|_{F}$ is a homeomorphism onto the open subset of $H^{(0)}$.}, and
		\item a continuous cocycle $c\colon G_F\to \T$
	\end{enumerate}
	which satisfy the following property :

	For each $U\in \Bis(G)$,
	the following diagram is commutative :
	
		\begin{center}
		\begin{tikzpicture}[auto]
		\node (a) at (0,0) {$C_0(U)$};
		\node (c) at (3,0){$C_0(\Phi(U\cap G_F))$};
		\node (d) at (0,-2) {$C_0(U\cap G_F)$};
		\draw[->] (a) to node {$\varphi_U$} (c) ;
		\draw[->,swap] (a) to node {$q_U$} (d);
		\draw[->,swap] (d) to node {$\varphi_{\Phi,c,U}$} (c);
		\end{tikzpicture}
		,
	\end{center}
	where $\varphi_U\colon C_0(U)\to C_0(\Phi(U\cap G_F))$ is the restriction $\varphi|_{C_0(U)}$,
	$q_U\colon C_0(U)\to C_0(U\cap G_F)$ is a surjective bounded linear map defined by $q_U(f)=f|_{U\cap G_F}$ for $f\in C_0(U)$ and $\varphi_{\Phi, c,U}\colon C_0(U\cap G_F)\to C_0(\Phi(U\cap G_F))$ is a linear isometric isomorphism defined by
	\[
	\varphi_{\Phi, c, U}(f)(\delta)\defeq c(\Phi^{-1}(\delta)) f(\Phi^{-1}(\delta))
	\]
	for $f\in C_0(U\cap G_F)$ and $\delta\in \Phi(U\cap G_F)$.
	
	Moreover,
	assume that there exists a *-homomorphism $\widetilde{\varphi}\colon C^*_r(G_F)\to C^*_r(H)$ with the following commutative diagram\footnote{Note that this $\widetilde{\varphi}$ is unique if it exists since $q$ is surjective.
	In Example \ref{ex widetilde varphi dose not exists},
	we give an example such that $\widetilde{\varphi}$ does not exist.} :
	
		\begin{center}
		\begin{tikzpicture}[auto]
		\node (a) at (0,0) {$C^*_r(G)$};
		\node (c) at (3,0){$C^*_r(H)$};
		\node (d) at (0,-2) {$C^*_r(G_F)$};
		\draw[->] (a) to node {$\varphi$} (c) ;
		\draw[->,swap] (a) to node {$q$} (d);
		\draw[->,swap] (d) to node {$\widetilde{\varphi}$} (c);
		\end{tikzpicture}
		,
	\end{center}
	where $q\colon C^*_r(G)\to C^*_r(G_F)$ denote the *-homomorphism induced by the restriction \[C_c(G)\ni f\mapsto f|_{G_F}\in C_c(G_F).\]
	Then the formula
	\[\widetilde{\varphi}(f)(\delta)=\sum_{\alpha\in \Phi^{-1}(\{\delta\})}c(\alpha) f(\alpha)\]
	holds for all $f\in C_c(G_F)$ and $\delta\in H$.
\end{thm}

In short, the first half of Theorem \ref{thm summary of main theorem } states that the local property of a given *-homomorphism $\varphi$ can be described in terms of underlying \'etale groupoids.
The latter half states that $\varphi$ itself can be described in terms of \'etale groupoids if there exists $\widetilde{\varphi}$.
In Example \ref{ex widetilde varphi dose not exists} and the text above it,
we will investigate the condition that $\widetilde{\varphi}$ exists.

First, we summarize standing assumptions in this subsection.
In the entirety of this subsection,
we assume that $G$ and $H$ are \'etale groupoids.
Moreover, we assume that $H$ is effective except for Lemma \ref{lem: C_0(U) is closed} and Lemma \ref{lem C_0(U) is subbimodule that consists of normalizers}.
In addition, we assume that $\varphi\colon C^*_r(G)\to C^*_r(H)$ is a *-homomorphism such that $\varphi(C_0(G^{(0)}))\subset C_0(H^{(0)})$.
Except for Lemma \ref{lem: C_0(U) is closed},
we assume that $\varphi(C_0(G^{(0)}))\subset C_0(H^{(0)})$ is an ideal.

Since $\ker\varphi \cap C_0(G^{(0)})$ is an ideal of $C_0(G^{(0)})$,
there exists a closed subset $F\subset G^{(0)}$ such that $C_0(G^{(0)}\setminus F)=\ker\varphi \cap C_0(G^{(0)})$ holds.
By \cite[Lemma 10.3.1]{asims},
$F$ is an invariant set of $G$ and therefore $G_F\defeq d^{-1}(F)$ is a closed subgroupoid of $G$.

For $U\in \Bis(G)$,
recall that $C_0(U)$ is a $C_0(G^{(0)})$-subbimodule of $C^*_r(G)$ in the natural way.
We put $\varphi_U\defeq \varphi|_{C_0(U)}$.

\begin{lem}\label{lem: C_0(U) is closed}
	Let $G$ and $H$ be \'etale groupoids and $U\in \Bis(G)$.
	Assume that a *-homomorphism $\varphi\colon C^*_r(G)\to C^*_r(H)$ satisfies $\varphi(C_0(G^{(0)}))\subset C_0(H^{(0)})$.
	Then there exists an isometric linear map $\widetilde{\varphi_U}\colon C_0(U\cap G_F)\to C_0(H^{(0)})$ that makes the following diagram commutative : 
	\begin{center}
		\begin{tikzpicture}[auto]
		\node (a) at (0,0) {$C_0(U)$};
		\node (c) at (3,0){$\varphi(C_0(U))$};
		\node (d) at (0,-2) {$C_0(U\cap G_F)$};
		\draw[->] (a) to node {$\varphi_U$} (c) ;
		\draw[->,swap] (a) to node {$q_U$} (d);
		\draw[->,swap] (d) to node {$\widetilde{\varphi_U}$} (c);
		\end{tikzpicture}
		,
	\end{center}
	where $q_U\colon C_0(U)\to C_0(U\cap G_F)$ denotes the surjective bounded linear map defined by the restriction.
	In particular, $\varphi(C_0(U))\subset C^*_r(H)$ is a closed linear subspace.
\end{lem}

\begin{proof}
	
	Recall that $F\subset G^{(0)}$ is a closed invariant subset of $G$ such that $C_0(G^{(0)}\setminus F)=\ker\varphi \cap C_0(G^{(0)})$.
	We claim that $\lVert \varphi(m)\rVert= \sup_{\alpha\in U\cap G_F}\lvert m(\alpha)\rvert$ holds for all $m\in C_0(U)$.
	Since $U$ is a bisection,
	$m$ is a normalizer for $C_0(G^{(0)})$ by Proposition \ref{prop open support of normaliser is a bisection}.
	Hence $\varphi(m^*m)\in C_0(H^{(0)})$ follows from $m^*m\in C_0(G^{(0)})$ and $\varphi(C_0(G^{(0)}))\subset C_0(H^{(0)})$.
	By the definition of $F\subset G^{(0)}$,
	there exists an injective *-homomorphism $\widetilde{\varphi_{G^{(0)}}}\colon C_0(F)\to C_0(H^{(0)})$ that makes the following diagram commutative : 
	\begin{center}
		\begin{tikzpicture}[auto]
		\node (a) at (0,0) {$C_0(G^{(0)})$};
		\node (c) at (3,0){$C_0(H^{(0)})$};
		\node (d) at (0,-2) {$C_0(F)$};
		\draw[->] (a) to node {$\varphi_{G^{(0)}}$} (c) ;
		\draw[->,swap] (a) to node {$q_{G^{(0)}}$} (d);
		\draw[->,swap] (d) to node {$\widetilde{\varphi_{G^{(0)}}}$} (c);
		\end{tikzpicture}
		,
	\end{center}
	where $\varphi_{G^{(0)}}$ is the restriction of $\varphi$ to $C_0(G^{(0)})$ and $q_{G^{(0)}}\colon C_0(G^{(0)})\to C_0(F)$ denotes the *-homomorphism defined by the restriction.
	Hence, we obtain
	\[\lVert \varphi(m)\rVert^2=\lVert \varphi(m^*m)\rVert=\lVert q_{G^{(0)}}(m^*m)\rVert=\sup_{x\in F}\lvert m^*m(x) \rvert=\sup_{\alpha\in U\cap G_F}\lvert m(\alpha)\rvert^2\]
	and therefore $\lVert \varphi(m)\rVert= \sup_{\alpha\in U\cap G_F}\lvert m(\alpha)\rvert$.
	
	Now, we obtain an isometry $\widetilde{\varphi_U}\colon C_0(U\cap G_F)\to \varphi(C_0(U))$ that makes the following diagram commutative : 
	\begin{center}
		\begin{tikzpicture}[auto]
		\node (a) at (0,0) {$C_0(U)$};
		\node (c) at (3,0){$\varphi(C_0(U))$};
		\node (d) at (0,-2) {$C_0(U\cap G_F)$};
		\draw[->] (a) to node {$\varphi_U$} (c) ;
		\draw[->,swap] (a) to node {$q_U$} (d);
		\draw[->,swap] (d) to node {$\widetilde{\varphi_U}$} (c);
		\end{tikzpicture}
		,
	\end{center}
	where $q_U\colon C_0(U)\to C_0(U\cap G_F)$ denotes the surjective bounded linear map defined by the restriction.
	Since $C_0(U\cap G_F)$ is complete and $\widetilde{\varphi_U}$ is an isometry,
	$\varphi(C_0(U))=\widetilde{\varphi_U}(C_0(U\cap G_F))$ is a closed linear subspace of $C^*_r(H)$.
	\qed
		
\end{proof}

In the rest of this subsection,
we assume that $\varphi(C_0(G^{(0)}))\subset C_0(H^{(0)})$ is an ideal of $C_0(H^{(0)})$.

\begin{lem}\label{lem C_0(U) is subbimodule that consists of normalizers}
	Assume that $\varphi(C_0(G^{(0)}))\subset C_0(H^{(0)})$ is an ideal of $C_0(H^{(0)})$.
	Then $\varphi(C_0(U))\subset C^*_r(H)$ is a $C_0(H^{(0)})$-subbimodule and \[\varphi(C_0(U))\subset N(C^*_r(H), C_0(H^{(0)}))\] holds for all $U\in \Bis(G)$.
\end{lem}

\begin{proof}
	
	Take $m\in C_0(U)$ and $a, b\in C_0(H^{(0)})$.
	Let $\{e_i\}_{i\in I}\subset C_0(G^{(0)})$ be an approximate identity for $C^*_r(G)$.
	Since we assume that $\varphi(C_0(G^{(0)}))$ is an ideal of $C_0(H^{(0)})$,
	$a\varphi(e_i)$ and $\varphi(e_i)b$ are contained in $\varphi(C_0(G^{(0)}))$ for all $i\in I$.
	There exists $f,g\in C_0(G^{(0)})$ such that $\varphi(f)=a\varphi(e_i)$ and $\varphi(g)=\varphi(e_i)b$ hold.
	Now we have \[a\varphi(e_i)\varphi(m)\varphi(e_i)b=\varphi(fmg)\in \varphi(C_0(U))\]
	since $C_0(U)\subset C^*_r(G)$ is a $C_0(G^{(0)})$-subbimodule.
	Since $\{e_ime_i\}_{i\in I}$ converges to $m$,
	we obtain $a\varphi(m)b\in \varphi(C_0(U))$ by Lemma \ref{lem: C_0(U) is closed}.
	Hence, $\varphi(C_0(U))$ is a $C_0(H^{(0)})$-subbimodule.

	Next, we show $\varphi(C_0(U))\subset N(C^*_r(H), C_0(H^{(0)}))$.
	Take $m\in C_0(U)$ and $h\in C_0(H^{(0)})$.
	Let $\{e_i\}_{i\in I}\subset C_0(G^{(0)})$ be an approximate identity for $C^*_r(G)$.
	Since $\varphi(e_i)h\in \varphi(C_0(G^{(0)}))$ and $m\in N(C^*_r(G),C_0(G^{(0)}))$,
	we have $\varphi(m)\varphi(e_i)h\varphi(m^*)\in \varphi(C_0(G^{(0)}))(\subset C_0(H^{(0)}))$.
	Hence we obtain $\varphi(m)h\varphi(m)^*\in C_0(H^{(0)})$.
	One can show $\varphi(m)^*h\varphi(m)\in C_0(H^{(0)})$ in the same way.
	Therefore, $\varphi(m)$ is a normalizer for $C_0(H^{(0)})$ and we obtain $\varphi(C_0(U))\subset N(C^*_r(H), C_0(H^{(0)}))$.
	\qed
\end{proof}

In the rest of this subsection,
we assume that an \'etale groupoid $H$ is effective.

\begin{lem}\label{lem definition of the psi}
	Assume that $\varphi(C_0(G^{(0)}))\subset C_0(H^{(0)})$ is an ideal and $H$ is effective.
	Then there exists a semigroup homomorphism $\psi\colon \Bis(G)\to \Bis(H)$ such that $\varphi(C_0(U))=C_0(\psi(U))$ holds for all $U\in \Bis(G)$.
\end{lem}

\begin{proof}
	Take $U\in \Bis(G)$.
	Then $\varphi(C_0(U))\subset C^*_r(H)$ is a closed $C_0(H^{(0)})$-subbimodule that consists of normalizers for $C_0(H^{(0)})$ by Lemma \ref{lem: C_0(U) is closed} and Lemma \ref{lem C_0(U) is subbimodule that consists of normalizers}.
	Hence, there exists a unique $\psi(U)\in\Bis(H)$ such that $\varphi(C_0(U))=C_0(\psi(U))$ holds by Theorem \ref{thm homomorphism from Bis to slices is isomorphism}.
	
	We show the map $\psi\colon \Bis(G)\to \Bis(H)$ is a semigroup homomorphism.
	Take $U_1,U_2\in\Bis(G)$.
	Then we have 
	\begin{align*}
	&C_0(\psi(U_1)\psi(U_2))=C_0(\psi(U_1))C_0(\psi(U_2))\\
	&=\varphi(C_0(U_1))\varphi(C_0(U_2))=\varphi(C_0(U_1U_2))\\
	&= C_0(\psi(U_1U_2)).
	\end{align*}
	Thus, we obtain $\psi(U_1)\psi(U_2)=\psi(U_1U_2)$ and $\psi$ is a semigroup homomorphism.
	\qed
\end{proof}

	We let $\psi\colon \Bis(G)\to \Bis(H)$ denote the semigroup homomorphism defined in Lemma \ref{lem definition of the psi}.
	
\begin{lem}

	Put \[T\defeq \{U\cap G_F\in \Bis(G_F)\mid U\in \Bis(G)\}. \]
	Then $T$ is an inverse subsemigroup of $\Bis(G_F)$ and there exists a subsemigroup homomorphism $\widetilde{\psi}\colon T\to \Bis(H)$ that makes the following diagram commutative : 
	
	\begin{center}
		\begin{tikzpicture}[auto]
		\node (a) at (0,0) {$\Bis(G)$};
		\node (c) at (3,0){$\Bis(H)$};
		\node (d) at (0,-2) {$T$};
		\draw[->] (a) to node {$\psi$} (c) ;
		\draw[->,swap] (a) to node {$Q$} (d);
		\draw[->,swap] (d) to node {$\widetilde{\psi}$} (c);
		\end{tikzpicture}
		,
	\end{center}
where $Q\colon \Bis(G)\to T$ is the semigroup homomorphism defined by $Q(U)=U\cap G_F$ for each $U\in \Bis(G)$.
\end{lem}

\begin{proof}
	One can show that $T$ is an inverse subsemigroup of $\Bis(G_F)$ by straightforward calculations.
	It is sufficient to show that $\psi(U_1)=\psi(U_2)$ holds for $U_1,U_2\in \Bis(G)$ which satisfy $U_1\cap G_F=U_2\cap G_F$.
	It follows that $U_1\cap G_F=U_1\cap U_2\cap G_F$ from $U_1\cap G_F=U_2\cap G_F$.
	By Lemma \ref{lem: C_0(U) is closed},
	there exists an isometry $\widetilde{\varphi_{U_1}}\colon C_0(U_1\cap G_F)\to \varphi(C_0(U_1))$ that makes the following diagram commutative : 
	\begin{center}
		\begin{tikzpicture}[auto]
		\node (a) at (0,0) {$C_0(U_1)$};
		\node (c) at (3,0){$\varphi(C_0(U_1))$};
		\node (d) at (0,-2) {$C_0(U_1\cap G_F)$};
		\draw[->] (a) to node {$\varphi_{U_1}$} (c) ;
		\draw[->,swap] (a) to node {$q_{U_1}$} (d);
		\draw[->,swap] (d) to node {$\widetilde{\varphi_{U_1}}$} (c);
		\end{tikzpicture}
		,
	\end{center}
	where $q_{U_1}\colon C_0(U_1)\to C_0(U_1\cap G_F)$ denotes the surjective bounded linear map defined by the restriction.
	Now we have 
	\begin{align*}
	\varphi(C_0(U_1))&=\widetilde{\varphi_{U_1}}(C_0(U_1\cap G_F))=\widetilde{\varphi_{U_1}}(C_0(U_1\cap U_2\cap G_F))\\
	&=\widetilde{\varphi_{U_1}}(q_{U_1}(C_0(U_1\cap U_2)))=\varphi(C_0(U_1\cap U_2)).
	\end{align*}
	Thus we obtain $\psi(U_1)=\psi(U_1\cap U_2)$.
	Replacing $U_1$ with $U_2$,
	we obtain $\psi(U_2)=\psi(U_1\cap U_2)$ and therefore $\psi(U_1)=\psi(U_2)$.
	\qed
\end{proof}

Since we assume that $\varphi(C_0(G^{(0)}))\subset C_0(H^{(0)})$ is an ideal,
there exists an open set $V\subset H^{(0)}$ such that $\varphi(C_0(G^{(0)}))=C_0(V)$.
In addition, since $F$ satisfies $C_0(G^{(0)}\setminus F)=\ker \varphi \cap C_0(G^{(0)})$,
there exists  a *-homomorphism $\widetilde{\varphi_{G^{(0)}}}\colon C_0(F)\to C_0(V)$ that makes the following diagram commutative : 
\begin{center}
	\begin{tikzpicture}[auto]
	\node (a) at (0,0) {$C_0(G^{(0)})$};
	\node (c) at (3,0){$C_0(V)$};
	\node (d) at (0,-2) {$C_0(F)$};
	\draw[->] (a) to node {$\varphi_{G^{(0)}}$} (c) ;
	\draw[->,swap] (a) to node {$q_{G^{(0)}}$} (d);
	\draw[->,swap] (d) to node {$\widetilde{\varphi_{G^{(0)}}}$} (c);
	\end{tikzpicture}
	,
\end{center}
where $q_{G^{(0)}}$ denotes the *-homomorphism defined by the restriction.
One can see that $\widetilde{\varphi_{G^{(0)}}}$ is indeed an *-isomorphism.
By Gelfand-Naimark duality,
there exists a homeomorphism $\widehat{\varphi_{G^{(0)}}}\colon V\to F$ such that $\widetilde{\varphi_{G^{(0)}}}(f)(y)=f(\widehat{\varphi_{G^{(0)}}}(y))$ holds for all $f\in C_0(F)$ and $y\in V$.
Put $\sigma\defeq\widehat{\varphi_{G^{(0)}}}^{-1}\colon F\to V\subset H^{(0)}$.
We regard the range of $\sigma$ as $H^{(0)}$ rather than $V$.

Now, we obtain a semigroup homomorphism $\widetilde{\psi}\colon T\to \Bis(H)$ and a homeomorphism $\sigma\colon F\to H^{(0)}$.

\begin{lem}\label{lem sigma is an equivalent map}
	The above maps $\sigma\colon F\to H^{(0)}$ and $\widetilde{\psi}\colon T\to \Bis(H)$ satisfy the condition in Proposition \ref{proposition equivariant map induces groupoid homomorphism} for the canonical actions of bisections $T\curvearrowright F$ and $\Bis(H)\curvearrowright H^{(0)}$.
\end{lem}
\begin{proof}
	Take $x\in F$ and $U\in \Bis(G)$ such that $x\in d(U\cap G_F)$.
	First, we show that  $\sigma(x)\in d(\widetilde{\psi}(U\cap G_F))$.
	We have
	\begin{align*}
	C_0(\sigma(d(U\cap G_F)))&=C_0(\widehat{\varphi_{G^{(0)}}}^{-1}(d(U\cap G_F))) \\
	&=\widetilde{\varphi_{G^{(0)}}}(C_0(d(U\cap G_F))) = C_0(\widetilde{\psi}(d(U\cap G_F))).
	\end{align*}
	
	Note that we use the condition that $\widehat{\varphi_{G^{(0)}}}$ is injective to deduce the second equality.
	Thus, we obtain $\sigma(d(U\cap G_F))=\widetilde{\psi}(d(U\cap G_F))$.
	Since $\widetilde{\psi}$ is a semigroup homomorphism,
	we have $\widetilde{\psi}(d(U\cap G_F))=d(\widetilde{\psi}(U\cap G_F))$.
	Therefore we obtain \[\sigma(x)\in \sigma(d(U\cap G_F))= d(\widetilde{\psi}(U\cap G_F)).\]
	
	Since we assume that $x\in d(U\cap G_F)$,
	there exists $\alpha\in U\cap G_F$ such that $d(\alpha)=x$.
	In addition, since $\sigma(x)\in d(\widetilde{\psi}(U\cap G_F))$,
	there exists $\delta\in \widetilde{\psi}(U\cap G_F)$ such that $d(\delta)=\sigma(x)$.
	In order to complete the proof,
	it is sufficient to show $\sigma(r(\alpha))=r(\delta)$.
	Instead, we show that $r(\alpha)=\widehat{\varphi_{G^{(0)}}}(r(\delta))$.
	Take $n\in C_0(U)$ such that $n(\alpha)\not=0$.
	For all $f\in C_0(G^{(0)})$,
	we have
	\begin{align*}
	\lvert n(\alpha)\rvert ^2 f(r(\alpha))&=n^*fn(d(\alpha))=q_{G^{(0)}}(n^*fn)(x)=q_{G^{(0)}}(n^*fn)(\widehat{\varphi_{G^{(0)}}}(d(\delta))) \\
	&=\widetilde{\varphi_{G^{(0)}}}(q_{G^{(0)}}(n^*fn))(d(\delta))=\varphi_{G^{(0)}}(n^*fn)(d(\delta))\\
	&=\varphi(n)^*\varphi(f)\varphi(n)(d(\delta))=\varphi(n^*n)(d(\delta))\varphi(f)(r(\delta))\\
	&= \varphi(n^*n)(\sigma(x)) \widetilde{\varphi_{G^{(0)}}}(q_{G^{(0)}}(f))(r(\delta))\\
	&=n^*n(x) q_{G^{(0)}}(f)(\widehat{\varphi_{G^{(0)}}}(r(\delta)))=\lvert n(\alpha)\rvert ^2  f(\widehat{\varphi_{G^{(0)}}}(r(\delta))).
	\end{align*}
	Hence, $f(r(\alpha))=f(\widehat{\varphi_{G^{(0)}}}(r(\delta)))$ holds for all $f\in C_0(G^{(0)})$.
	Therefore we obtain $r(\alpha)=\widehat{\varphi_{G^{(0)}}}(r(\delta))$ by Urysohn's lemma.
	\qed
\end{proof}

By Lemma \ref{lem sigma is an equivalent map} and Proposition \ref{proposition equivariant map induces groupoid homomorphism},
we obtain the groupoid homomorphism $\Phi\colon T\ltimes F\to \Bis(H)\ltimes H^{(0)}$.
Since $\Phi|_{F}\colon F\to H^{(0)}$ is a homeomorphism onto its image,
$\Phi$ is locally homeomorphic.
Since $T\ltimes F$ and $\Bis(H)\ltimes H^{(0)}$ are isomorphic to $G_F$ and $H$ respectively,
we obtain the groupoid homomorphism from $G_F$ to $H$ and denote it by $\Phi$ again.
This $\Phi$ is given explicitly as follows.
Take $\alpha\in G_F$ and $U\in \Bis(G)$ with $\alpha\in U$.
Then there exists $\alpha'\in \psi(U)$ such that $\sigma(d(\alpha))=d(\alpha')$.
This $\alpha'$ is nothing but $\Phi(\alpha)$.
In the proof of Lemma \ref{lem sigma is an equivalent map},
we obtained $\sigma(d(U\cap G_F))=\widetilde{\psi}(d(U\cap G_F))$.
In addition, we have $\widetilde{\psi}(d(U\cap G_F))=d(\widetilde{\psi}(U\cap G_F))=d(\psi(U))$.
Therefore, we have the following commutative diagram : 
\begin{center}
	\begin{tikzpicture}[auto]
	\node (a) at (0,0) {$U\cap G_F$};
	\node (c) at (3,0){$\psi(U)$};
	\node (d) at (0,-2) {$d(U\cap G_F)$};
	\node (e) at (3, -2) {$d(\psi(U))$};
	\draw[->] (a) to node {$\Phi$} (c) ;
	\draw[->,swap] (a) to node {$d$} (d);
	\draw[->,swap] (c) to node {$d$} (e);
	\draw[->] (d) to node {$\sigma$} (e);
	\end{tikzpicture}
	.
\end{center}
In particular,
$\Phi$ gives a homeomorphism from $U\cap G_F$ to $\psi(U)$ since the vertical domain maps and $\sigma$ are homeomorphisms.
Note that we have $\Phi|_{F}=\sigma=\widehat{\varphi_{G^{(0)}}}^{-1}$.

\begin{lem}\label{lem: well-definedness of the cocycle}
	Fix $\alpha\in G_F$.
	Take $U\in \Bis(G)$ and $n\in C_0(U)$ such that $\alpha \in U$ and $n(\alpha)\not=0$ hold.
	Then $\lvert\varphi(n)(\Phi(\alpha))\rvert=\lvert n(\alpha)\rvert$ holds.
	Moreover,
	the value \[ \frac{\varphi(n)(\Phi(\alpha))}{n(\alpha)}\in \T\] is independent of the choice of $U$ and $n$.
\end{lem}

\begin{proof}
	First, we show $\lvert\varphi(n)(\Phi(\alpha))\rvert=\lvert n(\alpha)\rvert$.
	This follows from the next calculation : 
	\begin{align*}
	\lvert\varphi(n)(\Phi(\alpha))\rvert^2&=\varphi(n^*n)(d(\Phi(\alpha)))=\varphi(n^*n)(\Phi(d(\alpha))) \\
	&=\widetilde{\varphi_{G^{(0)}}}(q_{G^{(0)}}(n^*n))(\Phi(d(\alpha)))=q_{G^{(0)}}(n^*n)(d(\alpha))=\lvert n(\alpha)\rvert^2.
	\end{align*}
		
	Next, we show that the value \[ \frac{\varphi(n)(\Phi(\alpha))}{n(\alpha)}\in \T\] is independent of the choice of $U$ and $n$.
	Consider $W\in \Bis(G)$ and $m\in C_0(W)$ such that $\alpha\in W$ and $m(\alpha)\not=0$ hold.
	Then we have
	\[
	n^*m(d(\alpha))=\overline{n(\alpha)}m(\alpha)
	\]
	and 
	\[
	\varphi(n^*m)(\Phi(d(\alpha)))=\varphi(n^*m)(d(\Phi(\alpha)))=\overline{\varphi(n)(\Phi(\alpha))}\varphi(m)(\Phi(\alpha)).
	\]
	Take $f\in C_c(d(U\cap W))$ such that $f(d(\alpha))=1$.
	Then we have $nf, mf\in C_c(U\cap W)$ and $(nf)^*(mf)\in C_0(G^{(0)})$.
	Combining with
	\begin{align*} n^*m(d(\alpha))&=(nf)^*(mf)(d(\alpha))=\varphi((nf)^*(mf))(\Phi(d(\alpha)))\\
	&=\overline{\varphi(n)(\Phi(\alpha))}\varphi(m)(\Phi(\alpha))
	\end{align*}
	and $\varphi(n)(\Phi(\alpha))/n(\alpha)\in\T$,
	we obtain 
	\[
	\frac{\varphi(m)(\Phi(\alpha))}{m(\alpha)}=\overline{\frac{n(\alpha)}{\varphi(n)(\Phi(\alpha))}}=\frac{\varphi(n)(\Phi(\alpha))}{n(\alpha)}.
	\]
	Thus, the value \[ \frac{\varphi(n)(\Phi(\alpha))}{n(\alpha)}\in \T\] is independent of the choice of $U$ and $n$.
	\qed
\end{proof}

\begin{lem}\label{lem definition of cocyle}
	Define $c\colon G_F\to \T$ by
	\[
	c(\alpha)\defeq \frac{\varphi(n)(\Phi(\alpha))}{n(\alpha)},
	\]
	where $n\in C_0(U)$ and $U\in \Bis(G)$ satisfies $\alpha\in U$ and $n(\alpha)\not=0$\footnote{Note that $c$ is well-defined by Lemma \ref{lem: well-definedness of the cocycle}.}.
	Then $c\colon G_F\to \T$ is a continuous groupoid homomorphism. 
\end{lem}

\begin{proof}
	First, we show that $c$ is continuous.
	Fix $\alpha\in G_F$.
	Take $U\in \Bis(G)$ and $n\in C_0(U)$ so that $\alpha\in U$ and $n(\alpha)\not=0$ holds.
	Then we have 
	\[
	c(\gamma)=\frac{\varphi(n)(\Phi(\gamma))}{n(\gamma)}
	\]
	for all $\gamma\in \supp^{\circ}(n)$.
	Since $\varphi(n)\circ\Phi$ and $n$ are continuous,
	$c$ is also continuous at $\alpha$.
	Since $\alpha\in G_F$ is arbitrary,
	$c$ is continuous on $G_F$.
	
	Next, we show that $c$ is a groupoid homomorphism.
	Fix $\alpha,\beta\in G_F$ with $d(\alpha)=r(\beta)$.
	Take $U_1,U_2\in \Bis(G)$, $n\in C_0(U_1)$ and $m\in C_0(U_2)$ so that $\alpha\in U_1$, $\beta\in U_2$,
	$n(\alpha)\not= 0$ and $m(\beta)\not=0$ hold.
	Then we have $nm\in C_0(UV)$ and $nm(\alpha\beta)=n(\alpha)m(\beta)\not=0$.
	Hence, we obtain
	\[
	c(\alpha\beta)=\frac{\varphi(nm)(\Phi(\alpha\beta))}{nm(\alpha\beta)}=\frac{\varphi(n)(\Phi(\alpha))\varphi(m)(\Phi(\beta))}{n(\alpha)m(\beta)}=c(\alpha)c(\beta).
	\]
	Here, the second equality follows from the fact that $\varphi(n)$ and $\varphi(m)$ are supported on bisections that contain $\Phi(\alpha)$ and $\Phi(\beta)$ respectively.
	Therefore, $c\colon G_F\to \T$ is a continuous groupoid homomorphism.
	\qed
\end{proof}
	
	Thus, we have obtained the locally homeomorphic groupoid homomorphism $\Phi\colon G_F\to H$ and the continuous cocycle $c\colon G_F\to \T$.
	From the remark beneath Lemma \ref{lem sigma is an equivalent map},
	$\Phi|_{U\cap G_F}$ is a homeomorphism onto its image for each $U\in\Bis(G)$.
	Now, we are ready to show the first half of Theorem \ref{thm summary of main theorem }.
	
	\begin{prop} \label{prop varphiU =varphic, Phi,U}
	Fix $U\in \Bis(G)$.
	Define a homeomorphism $\Phi_U\defeq \Phi|_{U\cap G_F} \colon U\cap G_F\to \Phi(U\cap G_F)$.
	In addition, define $\varphi_{\Phi,c,U}\colon C_0(U\cap G_F)\to C_0(\Phi(U\cap G_F))$ to be
	\[\varphi_{\Phi,c,U}(f)(\delta)\defeq c(\Phi_U^{-1}(\delta))f(\Phi_U^{-1}(\delta))\]
	for $f\in C_0(U\cap G_F)$ and $\delta\in \Phi(U\cap G_F)$.
	Then $\varphi_{\Phi,c,U}$ is a linear isometric isomorphism.
	Moreover, $\varphi_U=\varphi_{\Phi,c, U}\circ q_U$ and $\widetilde{\varphi_U}=\varphi_{\Phi,c, U}$ hold.
	\end{prop}
	
	\begin{proof}
		It follows that $\varphi_{\Phi,c,U}$ is a linear isometric isomorphism from the direct calculation.
		Note that the inverse map of $\varphi_{\Phi,c,U}$ is given by
		\[
		\varphi_{\Phi,c,U}^{-1}(g)(\alpha)=\overline{c(\alpha)}g(\Phi_U(\alpha))
		\]
		for $g\in C_0(\Phi(U\cap G_F))$ and $\alpha\in U\cap G_F$.
		
		We show $\varphi_U=\varphi_{\Phi,c, U}\circ q_U$.
		Take $g\in C_0(U)$ and $\delta\in \Phi(U\cap G_F)$.
		Assume that $g(\Phi_U^{-1}(\delta))\not=0$.
		Then we have
		\[
		c(\Phi_U^{-1}(\delta))=\frac{\varphi(g)(\Phi(\Phi_U^{-1}(\delta)))}{g(\Phi_U^{-1}(\delta))}=\frac{\varphi(g)(\delta)}{g(\Phi_U^{-1}(\delta))}=\frac{\varphi_U(g)(\delta)}{q_U(g)(\Phi^{-1}_U(\delta))}.
		\]
		Thus, we obtain
		\[
		\varphi_U(g)(\delta)=c(\Phi_U^{-1}(\delta))q_U(g)(\Phi^{-1}_U(\delta))=\varphi_{\Phi,c, U}\circ q_U(g)(\delta).
		\]
		If $g(\Phi_U^{-1}(\delta))=0$,
		we have 
		\begin{align*}
		\lvert \varphi_U(g)(\delta)\rvert^2&= \varphi(g^*g)(d(\delta))=g^*g(\Phi^{-1}(d(\delta)))\\
		&=g^*g(d(\Phi_U^{-1}(\delta)))=\lvert g(\Phi_U^{-1}(\delta))\rvert^2=0
		\end{align*}
		and therefore
		\[
		\varphi_{\Phi,c, U}\circ q_U(g)(\delta)=c(\Phi^{-1}_U(\delta))g(\Phi^{-1}_U(\delta))=0.
		\]
		Hence, we obtain \[\varphi_U(g)(\delta)=\varphi_{\Phi,c, U}\circ q_U(g)(\delta).\]
		for all $g\in C_0(U)$ and $\delta\in\Phi(U\cap G_F)$.
		The last assertion follows from the fact that we have $\varphi_U=\widetilde{\varphi_U}\circ q_U$ and $q_U$ is surjective.
		\qed
	\end{proof}
	
	Define a *-homomorphism $\varphi_{\Phi,c}\colon C_c(G_F)\to C_c(H)$ to be
	\[
	\varphi_{\Phi,c}(f)(\delta)\defeq \sum_{\alpha\in \Phi^{-1}(\delta)}c(\alpha)f(\alpha)
	\]
	for $f\in C_c(G_F)$ and $\delta\in H$.
	Note that the right hand side of the above formula is a finite sum since $\Phi$ is locally homeomorphic and the support of $f$ is compact.
	In addition,
	we use the condition that $\Phi$ is injective on $F$ to check that $\varphi_{\Phi,c}$ preserves the multiplications.
	Thus one can check that $\varphi_{\Phi,c}\colon C_c(G_F)\to C_c(H)$ is actually a *-homomorphism.
	We shall remark that $\varphi_{\Phi,c}$ does not always extend to $C^*_r(G_F)$ if $G_F$ is not amenable. 
	We will give a relevant example in Example \ref{ex widetilde varphi dose not exists}.
	
	Finally, we complete the proof of Theorem \ref{thm summary of main theorem }.
	
	\begin{prop}\label{prop decompotision of varphi}
		Assume that there exists a *-homomorphism 
		\[\widetilde{\varphi}\colon C^*_r(G_F)\to C^*_r(H)\]
		with the following commutative diagram :
		\begin{center}
			\begin{tikzpicture}[auto]
			\node (a) at (0,0) {$C^*_r(G)$};
			\node (c) at (3,0){$C^*_r(H)$};
			\node (d) at (0,-2) {$C^*_r(G_F)$};
			\draw[->] (a) to node {$\varphi$} (c) ;
			\draw[->,swap] (a) to node {$q$} (d);
			\draw[->,swap] (d) to node {$\widetilde{\varphi}$} (c);
			\end{tikzpicture}
			,
		\end{center}
	where $q\colon C^*_r(G)\to C^*_r(G_F)$ denote the *-homomorphism induced by the restriction.
	Then $\widetilde{\varphi}(f)=\varphi_{\Phi,c}(f)$ holds for all $f\in C_c(G_F)$.
	In particular, $\varphi_{\Phi,c}$ extends to $C^*_r(G_F)$ and $\widetilde{\varphi}=\varphi_{\Phi,c}$ holds.
	\end{prop}
	
	\begin{proof}
		Fix $U\in \Bis(G)$.
		By Proposition \ref{prop varphiU =varphic, Phi,U},
		we have $\widetilde{\varphi}(f)=\varphi_{\Phi,c}(f)$ for all $f\in C_c(U\cap G_F)$.
		Since $C_c(G_F)$ is the linear span of $\bigcup_{U\in \Bis(G)}C_c(U\cap G_F)$,
		we obtain $\widetilde{\varphi}(f)=\varphi_{\Phi,c}(f)$ for all $f\in C_c(G_F)$.
		\qed 
	\end{proof}

	Now, we have completed the proof of Theorem \ref{thm summary of main theorem }.
	In the last of this subsection,
	we give some remarks about our results in this subsection.
	First, we present a special case of Theorem \ref{thm summary of main theorem }.
	Applying Theorem \ref{thm summary of main theorem } for surjective *-homomorphisms,
	we obtain quotients of underlying \'etale groupoids (Corollary \ref{cor : surjective *-hom induces a quotient of etale groupoid}).
	We refer the readers to \cite[Section 3]{komura_2020} for the quotient \'etale groupoids.
	We remark that the quotient \'etale groupoid $G/{\Iso(G)^{\circ}}$ appearing in Corollary \ref{cor : surjective *-hom induces a quotient of etale groupoid} is nothing but the groupoid of germs associated with $G$ in the sense of \cite[Proposition 3.2]{renault}.
	
	\begin{cor}\label{cor : surjective *-hom induces a quotient of etale groupoid}
		Let $G$ be an \'etale groupoid and $H$ be an \'etale effective groupoid.
		Assume that there exists a surjective *-homomorphism $\varphi\colon C^*_r(G)\to C^*_r(H)$ such that $\varphi(C_0(G^{(0)}))=C_0(H^{(0)})$.
		Then there exists a closed invariant subset $F\subset G^{(0)}$ such that $H$ is isomorphic to $G_F/{\Iso(G_F)^{\circ}}$ as \'etale groupoids.
		Moreover, if $\varphi$ is injective on $C_0(G^{(0)})$ (and therefore $C_0(G^{(0)})$ and $C_0(H^{(0)})$ are isomorphic),
		then $H$ is isomorphic to $G/{\Iso(G)^{\circ}}$.
	\end{cor}
	
	\begin{proof}
		Applying Theorem \ref{thm summary of main theorem },
		we obtain a closed invariant $F\subset G^{(0)}$ and a locally homeomorphic groupoid homomorphism $\Phi\colon G_F\to H$.
		Since we assume that $\varphi$ is surjective,
		it turns out that $\Phi\colon G_F\to H$ is also surjective.
		In addition,
		one can see $\Phi^{-1}(H^{(0)})=\Iso(G_F)^{\circ}$ since $H$ is effective.
		Therefore,
		$\Phi$ induces an isomorphism $G_F/\Iso(G_F)^{\circ}\simeq H$ by the fundamental theorem on homomorphisms \cite[Proposition 2.2]{komura_2022}.
		Now, the last assertion is obvious since $F$ coincides with $G^{(0)}$ if $\varphi$ is injective on $C_0(G^{(0)})$.
		\qed
	\end{proof}
		
		Combining Corollary \ref{cor : surjective *-hom induces a quotient of etale groupoid} and groupoid quotients,
		we obtain the following rigidity result (Corollary \ref{cor generalized rigidity result}) for not necessarily effective groupoids.
		
	\begin{cor}\label{cor generalized rigidity result}
		Let $G$ and $H$ be \'etale (not necessarily effective) groupoids and $Q\colon H\to H/{\Iso(H)^{\circ}}$ be the quotient map.
		Assume that there exists a surjective *-homomorphism $\varphi\colon C^*_r(G)\to C^*_r(H)$ such that $\varphi$ give an isomorphism between $C_0(G^{(0)})$ and $C_0(H^{(0)})$.
		In addition,
		assume that $\Iso(H)^{\circ}\subset H$ is closed and the *-homomorphism
		\[
		\varphi_{Q}\colon C_c(H)\ni f \mapsto \left(\delta \mapsto \sum_{\alpha\in Q^{-1}(\delta)} f(\alpha) \right)\in C_c(H/{\Iso(H)^{\circ}})
		\]
		extends to the *-homomorphism from $C^*_r(H)$ to $C^*_r(H/{\Iso(H)^{\circ}})$\footnote{This assumption holds if $H$ is amenable. See \cite[Definition 10.1.2]{asims} for the definition of amenable groupoids.}.
		Then $G/{\Iso(G)^{\circ}}$ is isomorphic to $H/{\Iso(H)^{\circ}}$.		
	\end{cor}
	
	\begin{proof}
		First, $H/\Iso(H)^{\circ}$ is Hausdorff since we assume that $\Iso(H)^{\circ}\subset H$ is closed (\cite[Proposition 3.11]{komura_2020}).
		We denote the extension of $\varphi_{Q}$ by $\varphi_{Q}\colon C^*_r(H)\to C^*_r(H/\Iso(H)^{\circ})$ again.
		Then $\varphi_{Q}$ is surjective and the restriction $\varphi_{Q}|_{C_0(H^{(0)})}$ is an isomorphism onto $C_0((H/\Iso(H)^{\circ})^{(0)})$ by \cite[Proposition 3.13 and Lemma 3.14]{komura_2020}.
		In addition,
		note that $H/\Iso(H)^{\circ}$ is effective.
		Therefore we may apply Corollary \ref{cor : surjective *-hom induces a quotient of etale groupoid} for $\varphi_{Q}\circ \varphi \colon C^*_r(G)\to C^*_r(H/\Iso(H)^{\circ})$ and this yields an isomorphism $G/\Iso(G)^{\circ}\simeq H/\Iso(H)^{\circ}$.
		\qed
	\end{proof}

	\begin{rem}
		We remark that the converse of Corollary \ref{cor generalized rigidity result} does not hold.
		For example,
		let $G\defeq \{e\}$ be the trivial group and $H\defeq \Z$.
		Then both of $G/\Iso(G)^{\circ}$ and $H/\Iso(H)^{\circ}$ are trivial groups although there does not exist a surjective *-homomorphism from $C^*_r(G)=\C$ to $C^*_r(H)=C(\T)$.
		
		From Corollary \ref{cor generalized rigidity result},
		we deduce that the quotient groupoids $G/\Iso(G)^{\circ}$ define an invariant for inclusions of C*-algebras $(C^*_r(G), C_0(G^{(0)}))$ associated with \'etale amenable groupoids $G$ such that $\Iso(G)^{\circ}\subset G$ are closed.
		Namely, for \'etale amenable groupoids $G$ and $H$ such that $\Iso(G)^{\circ}\subset G$ and $\Iso(H)^{\circ}\subset H$ are closed,
		$G/\Iso(G)^{\circ}$ and $H/\Iso(H)^{\circ}$ are isomorphic if there exists a *-isomorphism $\varphi\colon C^*_r(G)\to C^*_r(H)$ such that $\varphi(C_0(G^{(0)}))=C_0(H^{(0)})$.
		In particular,
		since the orbit spaces of $G$ and $G/\Iso(G)^{\circ}$ are homeomorphic,
		the orbit spaces $G^{(0)}/G$ also define an invariant for inclusions of C*-algebras $(C^*_r(G), C_0(G^{(0)}))$ associated with \'etale amenable groupoids $G$ with closed $\Iso(G)^{\circ}$.
		Note that $G^{(0)}/G$ is the quotient space of $G^{(0)}$ with respect to the equivalence relation defined by declaring $x\sim y$ if there exists $\alpha\in G$ such that $d(\alpha)=x$ and $r(\alpha)=y$.
	\end{rem}

	Next, we investigate the condition that $\widetilde{\varphi}$ in Proposition \ref{prop decompotision of varphi} exists.
	Recall that an \'etale groupoid $G$ is said to be inner exact if the following sequence
	\[
	\begin{tikzcd}
	0\arrow{r} & C^*_r(G_{G^{(0)}\setminus F})\arrow{r}{\iota} & C^*_r(G)\arrow{r}{q} & C^*_r(G_F)\arrow{r}& 0 
	\end{tikzcd}
	\]
	is exact for each closed invariant subset $F\subset G^{(0)}$,
	where $\iota$ is the inclusion map and $q$ is the restriction.
	If $G$ is inner exact,
	then $\widetilde{\varphi}$ in Proposition \ref{prop decompotision of varphi} exists.
	In particular,
	$\widetilde{\varphi}$ always exists for all amenable groupoids, since amenable groupoids are inner exact (see \cite[Definition 10.1.2, Theorem 10.1.4 and Proposition 10.3.2]{asims} for this fact and the amenability of \'etale groupoids).
	In addition,
	if $\varphi$ is injective on $C_0(G^{(0)})$,
	then $\widetilde{\varphi}$ exists and $\widetilde{\varphi}$ is nothing but $\varphi$ since $F=G^{(0)}$ holds.
	In particular, no matter whether $G$ is amenable or not, every *-automorphisms on $C^*_r(G)$ that preserves $C_0(G^{(0)})$ comes from groupoid automorphisms and continuous cocycles on $G$ (Corollary \ref{cor: AutG is a semidirect product}).
	
	On the other hand,
	Example \ref{ex widetilde varphi dose not exists} is an example such that $\widetilde{\varphi}$ does not exist,
	although the \'etale groupoid $G$ in Example \ref{ex widetilde varphi dose not exists} is not effective.
	To our knowledge, it is an open problem whether $\widetilde{\varphi}$ always exists or not under the assumption that $G$ is effective.
	
	\begin{ex}[HLS groupoid] \label{ex widetilde varphi dose not exists}
		We give an example where $\widetilde{\varphi}$ dose not exist.

		Let $G$ be the HLS groupoid associated with the free group $\Gamma\defeq F_2$ (see \cite{Willett} for details) and $H\defeq\{e\}$ be the trivial group.
		Then we have $C^*_r(H)=\C$.
		Recall that
		\[
		G=\coprod_{n\in\N\cup \{\infty\}}\Gamma_n\times\{n\},
		\]
		where $\Gamma_n$ is a finite quotient group of $\Gamma$ for $n\in\N$ and $\Gamma_{\infty}=\Gamma$ (see \cite{Willett} for the precise definition).
		Define a *-homomorphism $\varphi\colon C_c(G)\to \C$ by
		\[
		\varphi(f)\defeq \sum_{s\in\Gamma}f((s,\infty))
		\]
		for $f\in C_c(G)$.
		Then $\varphi$ extends to the *-homomorphism $\varphi\colon C^*_r(G)\to \C$ since $C^*_r(G)$ coincides with the universal groupoid C*-algebra $C^*(G)$ by \cite[Theorem 1.2]{Willett}.
		Note that the corresponding closed invariant set is $F=\{\infty\}$.
		Since $F_2$ is not amenable,
		$\widetilde{\varphi}\colon C^*_r(G_F)\to \C$ dose not exist by \cite[Theorem 2.6.8]{brown}.
	\end{ex}

	\begin{ex}
		From a *-homomorphism $\varphi\colon C^*_r(G)\to C^*_r(H)$ such that $\varphi(C_0(G^{(0)}))\subset C_0(H^{(0)})$ is an ideal,
		we have constructed a closed invariant subset $F\subset G^{(0)}$,
		a locally homeomorphic groupoid homomorphisms $\Phi\colon G_F\to H$ such that $\Phi|_{F}$ is injective and a continuous cocycle $c\colon G_F\to \T$.
		If $G$ is amenable,
		this construction is bijective.
		Indeed, for such $F\subset G^{(0)}$, $\Phi\colon G_F\to H$ and $c\colon G_F\to \T$,
		the *-homomorphism $\varphi_{\Phi,c}\colon C_c(G_F)\to C^*_r(H)$ in Proposition \ref{prop decompotision of varphi} extends to $C^*_r(G_F)$ by \cite[Theorem 9.2.3 and Theorem 10.1.4]{asims}.
		Thus, we obtain a *-homomorphism $\varphi\colon C^*_r(G)\to C^*_r(H)$ which induces $F, \Phi$ and $c$.
		If $G$ is not amenable, these constructions need not to be bijective.
		We give such an example.
		
		Let $\Gamma$ be a countably infinite discrete group and $X\defeq \Gamma\cup \{\infty\}$ be the one point compactification of $\Gamma$.
		The left multiplication $\Gamma\curvearrowright \Gamma$ extends to the action on $X$ and we denote it by $\sigma\colon \Gamma\curvearrowright X$.
		Putting $G\defeq \Gamma\ltimes_{\sigma} X$, then $G$ is an effective \'etale groupoid.
		In addition, $F\defeq \{\infty\}$ is a closed invariant subset of $G$.
		Let $H\defeq \{e\}$ be the trivial group,
		$\Phi\colon G_{F}\to H$ be the unique group homomorphism and $c=1$.
		If $\Gamma$ is not amenable,
		then there dose not exist a *-homomorphism $\varphi\colon C^*_r(G)\to C^*_r(H)$ that induces $F$, $\Phi$ and $c$.
		Actually, there dose not exist a nonzero *-homomorphism $\varphi\colon C^*_r(G)\to C^*_r(H)$ if $\Gamma$ is not amenable.
		If such $\varphi\colon C^*_r(G)\to C^*_r(H)$ exists,
		then one obtain a nonzero *-homomorphism $C^*_r(\Gamma)\to \C$ by composing $\varphi$ with the canonical inclusion $C^*_r(\Gamma)\to C(X)\rtimes \Gamma=C^*_r(G)$.
		This contradicts to the non-amenability of $\Gamma$ (see \cite[Theorem 2.6.8]{brown}, for example).
 	\end{ex}

\subsection{Automorphism groups of $C^*_r(G)$ that globally preserve $C_0(G^{(0)})$}\label{subsection automorphism groups of C*r(G)}

In this subsection,
we investigate automorphism groups of $C^*_r(G)$ that globally preserve $C_0(G^{(0)})$.

\begin{defi}
	Let $D\subset A$ be an inclusion of C*-algebras.
	We define $\Aut_D(A)$ and $\FAut_D(A)$ to be
	\begin{align*}
	\Aut_D(A)&\defeq \{\varphi\in \Aut(A)\mid \varphi(D)=D\}, \\
	\FAut_D(A)&\defeq \{\varphi\in \Aut(A)\mid \text{$\varphi(a)=a$ for all $a\in D$}\}.
	\end{align*}
\end{defi}
Note that $\Aut_D(A)$ and $\FAut_D(A)$ are subgroups of $\Aut(A)$ and we have $\FAut_D(A)\subset \Aut_D(A)$.

For an \'etale groupoid $G$,
recall that we define
\[
Z(G,\T)\defeq \{c\colon G\to \T\mid \text{$c$ is a continuous groupoid homomorphism.}\}.
\]
Then $Z(G,\T)$ is an abelian group with respect to the pointwise product.
In addition, $\Aut(G)$ naturally acts on $Z(G,\T)$.
Indeed, for $\Phi\in \Aut(G)$ and $c\in Z(G,\T)$,
$\Phi. c\in Z(G,\T)$ is defined by
\[
\Phi. c(\alpha)\defeq c(\Phi^{-1}(\alpha))
\]
for $\alpha\in G$.
In the next corollary,
we consider the semidirect product $\Aut(G)\ltimes Z(G,\T)$ with respect to this action.
Recall that the product of $\Aut(G)\ltimes Z(G,\T)$ is 
\[(\Phi_1, c_1)\cdot (\Phi_2, c_2)=(\Phi_1\circ\Phi_2, (\Phi_2^{-1}.c_1) \cdot c_2)\]
for $(\Phi_1, c_1), (\Phi_2, c_2)\in \Aut(G)\ltimes Z(G,\T)$.

\begin{cor}[{\cite[Proposition 5.7]{Matui2011}}]\label{cor: AutG is a semidirect product}
	Let $G$ be an effective \'etale groupoid.
	For $c\in Z(G,\T)$ and $\Phi\in \Aut(G)$,
	define $\varphi_{\Phi,c}\in\Aut(C^*_r(G))$ by
	\[
	\varphi_{\Phi,c}(f)(\alpha)\defeq c(\Phi^{-1}(\alpha))f(\Phi^{-1}(\alpha))
	\]
	for $f\in C_c(G)$ and $\alpha\in G$.
	Then $\varphi_{\Phi,c}\in\Aut_{C_0(G^{(0)})}(C^*_r(G))$ holds.
	Moreover,
	the map
	\[\Psi\colon \Aut(G)\ltimes Z(G,\T)\ni (\Phi,c)\mapsto \varphi_{\Phi,c}\in \Aut_{C_0(G^{(0)})}(C^*_r(G))\]
	is an isomorphism as groups.
\end{cor}

\begin{rem}
	The above result is almost same as \cite[Proposition 5.7 (1)]{Matui2011}.
	In the proof of \cite[Proposition 5.7]{Matui2011},
	Renault's theory about Cartan C*-subalgebras \cite{renault} is used.
	In this paper, we obtained the proof of \ref{cor: AutG is a semidirect product} in a slightly more direct way without Renault's theory.	
\end{rem}

\begin{proof}[\sc{Proof of Corollary \ref{cor: AutG is a semidirect product}}]
	We show that $\Psi$ is a group homomorphism.
	Take $(\Phi_1, c_1), (\Phi_2, c_2)\in \Aut(G)\ltimes Z(G,\T)$.
	For $f\in C_c(G)$ and $\gamma\in G$,
	we have
	\begin{align*}
	\varphi_{\Phi_1, c_1}&\circ \varphi_{\Phi_2,c_2}(f)(\gamma)=\varphi_{\Phi_1, c_1}(\varphi_{\Phi_2,c_2}(f))(\gamma) \\
	&=c_1(\Phi_1^{-1}(\gamma))\varphi_{\Phi_2,c_2}(f)(\Phi_1^{-1}(\gamma)) \\
	&=c_1(\Phi_1^{-1}(\gamma))c_2(\Phi_2^{-1}(\Phi_1^{-1}(\gamma)))f(\Phi_2^{-1}(\Phi_1^{-1}(\gamma))) \\
	&=(\Phi_2^{-1}.c_1)((\Phi_1\circ\Phi_2)^{-1}(\gamma))c_2((\Phi_1\circ\Phi_2)^{-1}(\gamma))f((\Phi_1\circ\Phi_2)^{-1}(\gamma)) \\
	&=\varphi_{\Phi_1\circ\Phi_2, (\Phi_2^{-1}.c)\cdot c_2}(f)(\gamma)=\varphi_{(\Phi_1, c_1)(\Phi_2, c_2)}(f)(\gamma).
	\end{align*}
	Thus $\Psi$ is a group homomorphism.
	By Proposition \ref{prop decompotision of varphi},
	$\Psi$ is surjective.
	To show that $\Psi$ is injective,
	take $(\Phi_1, c_1), (\Phi_2, c_2)\in \Aut(G)\ltimes Z(G,\T)$ and assume $\varphi_{\Phi_1, c_1}=\varphi_{\Phi_2,c_2}$.
	If $\Phi_1\not=\Phi_2$,
	then $\Phi_1^{-1}(\gamma)\not=\Phi_2^{-1}(\gamma)$ holds for some $\gamma\in G$.
	By Urysohn's lemma, there exists $f\in C_c(G)$ such that $f(\Phi_1^{-1}(\gamma))\not=0$ and $f(\Phi_2^{-1}(\gamma))=0$ hold.
	It follows that $\varphi_{\Phi_1,c_1}(f)(\gamma)\not=0$ and $\varphi_{\Phi_2, c_2}(f)(\gamma)=0$,
	which contradicts to $\varphi_{\Phi_1, c_1}=\varphi_{\Phi_2,c_2}$.
	Hence we obtain $\Phi_1=\Phi_2$.
	To show $c_1=c_2$,
	take $\gamma\in G$ and $f\in C_c(G)$ so that $f(\Phi_1^{-1}(\gamma))=1$.
	Then we have
	\[
	c_1(\gamma)=\varphi_{\Phi_1,c_1}(f)(\gamma)=\varphi_{\Phi_1,c_2}(f)(\gamma)=c_2(\gamma).
	\]
	Thus we obtain $c_1=c_2$.
	Therefore we have shown that $\Psi$ is injective.
	In conclusion, $\Psi$ is a group isomorphism.
	\qed
\end{proof}

Finally, we investigate $\FAut_{C_0(G^{(0)})}(C^*_r(G))$.

\begin{prop} \label{prop: if groupoid automorphism is identity on the unit space, then it is identity}
	Let $G$ be a topologically principal \'etale groupoid and $\Phi\colon G\to G$ be a continuous groupoid automorphism.
	Assume that $\Phi(x)=x$ holds for all $x\in G^{(0)}$.
	Then $\Phi=\id$.
\end{prop}
	\begin{proof}
		Put $A=\{x\in G^{(0)}\mid d^{-1}(x)\cap r^{-1}(x)=\{x\}\}$.
		Then $A\subset G^{(0)}$ is dense.
		Since $d\colon G\to G^{(0)}$ is an open map,
		$d^{-1}(A)\subset G$ is dense.
		We show that $\Phi(\alpha)=\alpha$ holds for all $\alpha\in d^{-1}(A)$.
		Note that we have
		\[d(\Phi(\alpha))=\Phi(d(\alpha))=d(\alpha)\]
		and $r(\Phi(\alpha))=r(\alpha)$.
		Thus $(\Phi(\alpha)^{-1},\alpha)$ is a composable pair and we have $\Phi(\alpha)^{-1}\alpha\in G_{d(\alpha)}\cap G^{d(\alpha)}$.
		Since $d(\alpha)\in A$,
		we obtain $\Phi(\alpha)=\alpha$.
		Since $d^{-1}(A)\subset A$ is dense,
		it follows that $\Phi(\alpha)=\alpha$ holds for all $\alpha\in G$.
		\qed
	\end{proof}
		
	\begin{cor}
			Let $G$ be a topologically principal \'etale groupoid and $\Phi_1, \Phi_2\colon G\to G$ be continuous groupoid automorphisms.
			If $\Phi_1|_{G^{(0)}}=\Phi_2|_{G^{(0)}}$,
			then $\Phi_1=\Phi_2$ holds.
	\end{cor}
	\begin{proof}
		Since $\Phi_1\circ \Phi_2^{-1}(x)=x$ holds for all $x\in G^{(0)}$,
		we obtain $\Phi_1\circ \Phi_2^{-1}=\id$ by Proposition \ref{prop: if groupoid automorphism is identity on the unit space, then it is identity}.
		Hence we obtain $\Phi_1=\Phi_2$.
		\qed
	\end{proof}

	\begin{cor}\label{cor Z(G,T) is isomorphic to FAUT}
		Let $G$ be a topologically principal \'etale groupoid.
		Fix $(\Phi, c)\in \Aut(G)\ltimes Z(G,\T)$.
		Then $\Phi=\id_G$ holds if and only if $\varphi_{\Phi,c}\in \FAut_{C_0(G^{(0)})}(C^*_r(G))$ holds.
		In particular,
		the restriction of $\Psi$ in Corollary \ref{cor: AutG is a semidirect product} gives a group isomorphism
		\[
		\Psi|_{Z(G,\T)}\colon Z(G,\T)\to \FAut_{C_0(G^{(0)})}(C^*_r(G))
		\]
		and $\FAut_{C_0(G^{(0)})}(C^*_r(G))$ is an abelian group.
	\end{cor}
	
	\begin{proof}
		It is clear that $\varphi_{\Phi,c}\in \FAut_{C_0(G^{(0)})}(C^*_r(G))$ holds if $\Phi=\id_{G}$.
		Assume $\varphi_{\Phi,c}\in \FAut_{C_0(G^{(0)})}(C^*_r(G))$.
		Then we have
		\[
		f(\Phi^{-1}(x))=\varphi_{\Phi,c}(f)(x)=f(x)
		\]
		for all $f\in C_0(G^{(0)})$ and $x\in G^{(0)}$.
		Hence, we obtain $\Phi^{-1}|_{G^{(0)}}=\id_{G^{(0)}}$ and therefore $\Phi|_{G^{(0)}}=\id_{G^{(0)}}$.
		By Proposition \ref{prop: if groupoid automorphism is identity on the unit space, then it is identity},
		we obtain $\Phi=\id_{G}$.
		Now, the last assertion is clear.
		\qed
		\end{proof}
	
	For a topological group $H$,
	we denote the abelianization of $H$ by $H^{\ab}$.
	Recall that $H^{\ab}$ is the quotient group of $H$ by the closure of commutator subgroup $\overline{[H,H]}$.
	In addition, we denote the quotient map by $\pi\colon H\to H^{\ab}$ in the next corollary.
	
	\begin{cor}\label{cor action with large fixed point algebra factors abelianization}
		Let $G$ be a topologically principal \'etale groupoid,
		$H$ be a topological group and $\sigma\colon H\curvearrowright C^*_r(G)$ be an action such that
		\[
		\ker\sigma\defeq \{s\in H\mid \alpha_s=\id_{C^*_r(G)}\}
		\]
		is closed in $H$.
		Assume that the fixed point algebra
		\[
		C^*_r(G)^{\sigma}\defeq \bigcap_{s\in H}\{x\in C^*_r(G)\mid \sigma_s(x)=x \}
		\]
		contains $C_0(G^{(0)})$.
		Then there exists an action $\widetilde{\sigma}\colon H^{\ab}\curvearrowright C^*_r(G)$ such that
		$\sigma_s=\widetilde{\sigma}_{\pi(s)}$ holds for all $s\in H$.
		
	\end{cor}
	\begin{proof}
		Since we assume $C_0(G^{(0)})\subset C^*_r(G)^{\sigma}$,
		we have \[\sigma_s\in \FAut_{C_0(G^{(0)})}(C^*_r(G))\] for all $s\in H$.
		Hence, it follows $[H, H]\subset \ker\sigma$ since $\FAut_{C_0(G^{(0)})}(C^*_r(G))$ is an abelian group by Corollary \ref{cor Z(G,T) is isomorphic to FAUT}.
		Since we assume $\ker\sigma$ is closed,
		we obtain $\overline{[H,H]}\subset \ker\sigma$.
		Now, the existence of $\widetilde{\sigma}$ follows from the fundamental group theory.
		\qed
	\end{proof}
	The assumption that $\ker\sigma\subset H$ is closed in Corollary \ref{cor action with large fixed point algebra factors abelianization} holds if we assume some continuity of $\sigma$.
	For example,
	if $\sigma$ is a strongly continuous action,
	then $\ker\sigma\subset H$ is closed.

\bibliographystyle{plain}
\bibliography{bunken}

\end{document}